\documentclass[11pt, letterpaper,reqno]{amsart}

\usepackage{amsmath}
\usepackage{amssymb}
\usepackage{amsfonts}
\usepackage{ dsfont }
\usepackage{cases}
\usepackage{cite}
\usepackage[hyphens]{url}
\usepackage{hyperref}
\usepackage{amsthm}
\usepackage{upgreek}
\usepackage{color}

\numberwithin{equation}{section}
\DeclareMathOperator*{\esssup}{ess\,sup}
\DeclareMathOperator*{\essinf}{ess\,inf}

\newtheorem{theorem}{Theorem}[section]
\newtheorem{corollary}{Corollary}[section]
\newtheorem{lemma}{Lemma}[section]

\newtheorem{definition}{Definition}[section]

\newtheorem{remark}{Remark}[section]
\newtheorem{example}{Example}[section]
\theoremstyle{definition}
\theoremstyle{remark}

\newcommand{\leref}{Lemma~\ref}

\newcommand{\exref}{Example~\ref}

\newcommand{\thref}{Theorem~\ref}

\newcommand{\Rho}{{\pmb\rho}}
\newcommand{\Tau}{{\pmb\tau}}
\newcommand{\E}{\mathbb{E}}
\newcommand{\T}{\mathcal{T}}
\newcommand{\bT}{\mathbb{T}}
\newcommand{\eps}{\epsilon}

\title[]{On Zero-sum Optimal Stopping Games}
\author[]{Erhan Bayraktar} \thanks{This research was supported in part by the National Science Foundation under grant DMS-1613170.}  
\address{Department of Mathematics, University of Michigan}
\email{erhan@umich.edu}
\author[]{Zhou Zhou}
\address{IMA, University of Minnesota}
\email{zhoux528@umn.edu}
\date{\today}
\begin{document}
\maketitle
\begin{abstract}
On a filtered probability space $(\Omega,\mathcal{F},P,\mathbb{F}=(\mathcal{F}_t)_{t=0,\dotso,T})$, we consider stopping games $\overline V:=\inf_{\Rho\in\bT^{ii}}\sup_{\tau\in\T}\E[U(\Rho(\tau),\tau)]$ and $\underline V:=\sup_{\Tau\in\bT^i}\inf_{\rho\in\T}\E[U(\rho,\Tau(\rho))]$ in discrete time, where $U(s,t)$ is $\mathcal{F}_{s\vee t}$-measurable instead of $\mathcal{F}_{s\wedge t}$-measurable as is assumed in the literature on Dynkin games, $\T$ is the set of stopping times, and $\bT^i$ and $\bT^{ii}$ are sets of mappings from $\T$ to $\T$ satisfying certain non-anticipativity conditions. We will see in an example that there is no room for stopping strategies in classical Dynkin games unlike the new stopping game we are introducing.
We convert the problems into an alternative Dynkin game, and show that $\overline V=\underline V=V$, where $V$ is the value of the Dynkin game. We also get optimal $\Rho\in\bT^{ii}$ and $\Tau\in\bT^i$ for $\overline V$ and $\underline V$ respectively.
\end{abstract}
\section{Introduction}
 On a filtered probability space $(\Omega,\mathcal{F},P,\mathbb{F}=(\mathcal{F}_t)_{t=0,\dotso,T})$,
let us consider the game
\begin{equation}\label{1}
\inf_\rho\sup_\tau\E U(\rho,\tau)\quad\text{and}\quad\sup_\tau\inf_\rho\E U(\rho,\tau),
\end{equation}
where $U(s,t)$ is $\mathcal{F}_{s \vee t}$-measurable.

First, we consider the case in which $\rho$ and $\tau$ in the above are stopping times. That is, let
\begin{equation}\label{eq1}
\overline A:=\inf_{\rho\in\T}\sup_{\tau\in\T}\E U(\rho,\tau)\quad\text{and}\quad\underline A:=\sup_{\tau\in\T}\inf_{\rho\in\T}\E U(\rho,\tau),
\end{equation}
where $\T$ is the set of $\mathbb{F}$-stopping times taking values in $\{0,\dotso,T\}$.
In the particular case when
\begin{equation}\label{eq2}
U(s,t)=f_s 1_{\{s<t\}}+g_t 1_{\{s\geq t\}}
\end{equation}
in which $f_t$ and $g_t$ are bounded $\mathbb{F}$-adapted processes,
the problem above is said to be a Dynkin game (see for example \cite{MR0241121} and \cite[Chapter VI-6]{Neveu}). It is well-known that if $f\geq g$ then $\overline A=\underline A$.\footnote{For example, let $f_t=g_t=t$. Then $U(s,t)=s\wedge t$. By setting $\rho\equiv 0$ we can see that $\overline A=\underline A=0$.} Note that $U$ in \eqref{eq2} is $\mathbb{F}_{s\wedge t}$-measurable, and thus Dynkin game ends at the minimum of $\rho$ and $\tau$. 
However, for a general $U$, it may be possible that $\overline A>\underline A$ even for some very natural choices of $U$. For example, consider
\begin{equation}\label{eq3}
U(s,t)=|s-t|.
\end{equation}
With this choice of $U$, the problem in \eqref{eq1} becomes deterministic, and it is easy to see that $\overline A=\lceil T/2 \rceil>0=\underline A$. As opposed to the Dynkin game in which $U$ is given by \eqref{eq2}, we can see that the game with $U$ given by $\eqref{eq3}$ has not ended when only one of the players has stopped, i.e., the payoff will be further affected by the player who stops later. Intuitively, the failure of the equality $\overline A=\underline A$ of the two game values in \eqref{eq1} is due to the fact that the minimizer in the first game (represented by $\overline A$) is weaker than the minimizer in the second game (represented by $\underline A$). The opposite holds for the maximizer. In other words, the inner player in either of the games knows the outer player's prefixed strategy but not vice versa. 

One could try to amend the above situation by strengthening the outer players in these two games by giving them more choices. We will let them use strategies rather than just stopping times (which as we shall see are the simplest possible form of strategies). 
That is, we consider the stopping games
\begin{equation}\notag
\inf_\Rho\sup_{\tau\in\T}\E[U(\Rho(\tau),\tau)]\quad\text{and}\quad\sup_{\Tau}\inf_{\rho\in\T} \E[U(\rho,\Tau(\rho))],
\end{equation}
where $\Rho(\cdot),\Tau(\cdot):\ \T\mapsto\T$ satisfy certain non-anticipativity conditions. In fact, it is more meaningful to state these games with using strategies for the outer players because the inner players' actions changes the natural filtration of the reward processes and the games are still continuing after those actions. Hence we need to let the outer players adjust their strategies given the actions of the inner players.

One possible definition of a class of non-anticipative stopping strategies (we denote the collection of them as $\bT^i$) would be that, $\Rho\in\bT^i$, if $\Rho: \T\mapsto\T$ satisfies
$$\text{either}\quad\Rho(\sigma_1)=\Rho(\sigma_2)\leq\sigma_1\wedge\sigma_2\quad\text{or}\quad \Rho(\sigma_1)\wedge\Rho(\sigma_2)>\sigma_1\wedge\sigma_2,\quad\forall\sigma_1,\sigma_2\in\T.$$
That is, either the outer player acts first and chooses just a stopping  time, or she waits and observes the inner player and starting at the very next step makes a decision knowing the action of the inner player. 
Each stopping time $\T\subset\bT^i$ can be regarded as a stopping strategy that is indifferent to all inner stopping times. 

With the above definition let us consider the game 
\begin{equation}\label{eq4}
\overline B:=\inf_{\Rho\in\bT^i}\sup_{\tau\in\T}\E[U(\Rho(\tau),\tau)]\quad\text{and}\quad\underline B:=\sup_{\Tau\in\bT^i}\inf_{\rho\in\T} \E[U(\rho,\Tau(\rho))].
\end{equation}
Now compared to \eqref{eq1}, the outer players in the above two games have more power as they have more choices ($\bT^i$ other than $\T$). 
However, even in this situation it may still be the case that
\begin{equation}\label{eq:ineqinB}
\overline B>\underline B,
\end{equation}
which shows that the outer players still do not have enough strength to make $\overline B=\underline B$. Below is an example showing that $\overline B>\underline B$.
\begin{example}\label{g1}
Let $T=1$ and $U(s,t)=1_{\{s\neq t\}},\ t=0,1$. Then there are only two elements, $\Rho^0$ and $\Rho^1$, in $\bT^i$, with $\Rho^0(0)=\Rho^0(1)=0$ and $\Rho^1(0)=\Rho^1(1)=1$. It can be shown that $\overline B=1$ and $\underline B=0$.
\end{example}

Another possible definition of non-anticipative stopping strategies (we denote the collection as $\bT^{ii}$) would be that, $\Rho\in\bT^{ii}$, if $\Rho: \T\mapsto\T$ satisfies
\begin{equation}\label{e12345}
\text{either}\quad\Rho(\sigma_1)=\Rho(\sigma_2)<\sigma_1\wedge\sigma_2\quad\text{or}\quad \Rho(\sigma_1)\wedge\Rho(\sigma_2)\geq\sigma_1\wedge\sigma_2,\quad\forall\sigma_1,\sigma_2\in\T.
\end{equation}
That is, either the outer player acts first strictly before the other player or it waits and observes the inner player and at that very moment makes a decision knowing the action of the inner player. 
Now consider
\begin{equation}\label{eq5}
\overline C:=\inf_{\Rho\in\bT^{ii}}\sup_{\tau\in\T}\E[U(\Rho(\tau),\tau)]\quad\text{and}\quad\underline C:=\sup_{\Tau\in\bT^{ii}}\inf_{\rho\in\T} \E[U(\rho,\Tau(\rho))].
\end{equation}
Since $\bT^i\subset\bT^{ii}$,\footnote{To wit, let $\Rho\in\mathbb{T}^i$ and take $\sigma_1,\sigma_2\in\mathcal{T}$. If $\Rho(\sigma_1)\wedge\Rho(\sigma_2)>\sigma_1\wedge\sigma_2$, then \eqref{e12345} is satisfied. Now assume $\Rho(\sigma_1)=\Rho(\sigma_2)\leq\sigma_1\wedge\sigma_2$. If the strict inequality holds, then the first part of \eqref{e12345} is satisfied. Otherwise, $\Rho(\sigma_1)=\Rho(\sigma_2)=\sigma_1\wedge\sigma_2$, which implies the second part of \eqref{e12345}.} the outer players here have more power compared to the case in \eqref{eq4}.   

It turns out that by using strategies in $\bT^{ii}$, the outer players in \eqref{eq5} have too much power now, in the sense that it is possible that
\begin{equation}\label{eq:diffenofCs}
\overline C<\underline C.
\end{equation}
We still use \exref{g1} as an example.

\begin{example}\label{g2}
Let $T=1$ and $U(s,t)=1_{\{s\neq t\}},\ t=0,1$. Then in this case $\bT^{ii}$ is the set of all the maps from $\T$ to $\T$. By letting $\Rho(0)=0$ and $\Rho(1)=1$, we have that $\overline C=0$. By Letting $\Tau(0)=1$ and $\Tau(1)=0$, we have that $\underline C=1$.
\end{example}

Observe that $\overline B=\underline C$ and $\underline B=\overline C$ in Examples \ref{g1} and \ref{g2}. In fact it is by no means a coincidence as we will see later in this paper. That is, we always have
\begin{equation}\label{eq6}
\overline V:=\inf_{\Rho\in\bT^{ii}}\sup_{\tau\in\T}\E[U(\Rho(\tau),\tau)]=\sup_{\Tau\in\bT^i}\inf_{\rho\in\T} \E[U(\rho,\Tau(\rho))]=:\underline V.
\end{equation}
 Using $\bT^{ii}$ for $\Rho$ (the outer player) in the first game and $\bT^i$ for $\Tau$ (the outer player) in the second game above can be thought of as striking a balance between ``not enough power'' in \eqref{eq4} and ``too much power'' in \eqref{eq5} for the outer player. 
An intuitive reason of $\overline V=\underline V$ in \eqref{eq6} is that, at each time period we designate the same player (here we choose ``sup'') to act first (if one interprets ``to stop'' and ``not to stop'' as the allowable set of actions). So this player (``sup'') can only take advantage of the other's (``inf's'') previous behavior (as opposed to ``inf'' taking advantage of ``sup's'' current behavior in addition). 

\begin{remark}
In the continuous-time case, we have also have \eqref{eq:ineqinB} and \eqref{eq:diffenofCs} in general (see Remark
2.1, in \cite{MR3503728}). But in order to solve this problem the latter paper assumes that the pay-off is right continuous along stopping times in the sense of
expectation as in \cite{Ko1}, which is the most relaxed assumption in the continuous optimal stopping time literature (see also \cite{Ko2,Ko3}), and a result the difference between
the two types of non-anticipativity conditions disappear, i.e. $\overline B=\underline B$ and $\overline C=\underline C$. Hence the discrete-time case is interesting, because no such technical condition is needed and one is able to observe the structure of the problem more clearly.
\end{remark}

Our new zero-sum stopping game, which can be considered as a \textbf{two stage} Dynkin game, also captures the ``game nature''--- the interaction between players (i.e., players can adjust their own strategies according to the other's behavior). For example, unlike in $\overline A$ in \eqref{eq1} where $\rho$ does not depend on $\tau$ (but $\tau$ depends on $\rho$), in $\overline V$, $\Rho$ depends on $\tau$ by the definition of $\bT^{ii}$ (and of course $\tau$ still depends on $\Rho$ as $\tau$ is the inner player).

Since \cite{MR0241121} Dynkin games have been studied extensively, and we refer to the survey paper \cite{zbMATH06169716} and the references therein. The stopping game we introduce here is more suitable for handing conflict since in general players take turns in playing the game and the game does not end when one of the players act. In other words, stopping games are not always \emph{duels}, which is what the usual Dynkin game models. In Section 3.2 and 3.3, we will provide two applications of our game, including a robust utility maximization problem involving two American options, and an example of competing companies choosing times to enter the market.

We can also relate our paper to the following interesting phenomenon observed in stochastic differential games: In a zero-sum game with one player doing the inside optimization using an open loop strategy and the other player doing outside optimization using an Elliott-Kalton (non-anticipative) strategy, it has been observed (see for example \cite{bardi-capuzzo}, Theorem 3.11) that the infimum and supremum can not be exchanged. For the game to have a value when one needs to change the strength of the players appropriately: the outside player should have closed loop strategies and the inside player should have open loop controls. (This was only proved analytically using a viscosity comparison.)  We are able to observe this in an optimal stopping problem for the first time. Moreover, we prove it directly using probabilistic techniques only.

The rest of the paper is organized as follows. In the next section, we introduce the setup and the main result. We provide three examples in Section 3. In Section 4, we give the proof of the main result. Finally we give some insight for the corresponding problems in continuous time in Section 5.  

\section{The setup and the main result}
Let $(\Omega,\mathcal{F},P)$ be a probability space, and $\mathbb{F}=(\mathcal{F}_t)_{t=0,\dotso,T}$ be the filtration enlarged by $P$-null sets, where $T\in\mathbb{N}$ is the time horizon. Let $U:\{0,\dotso,T\}\times\{0,\dotso,T\}\times\Omega\mapsto\mathbb{R}$, such that $U(s,t,\cdot)\in\mathcal{F}_{s\vee t}$. For simplicity, we assume that $U$ is bounded. Denote $\E_t[\cdot]$ for $\E[\cdot|\mathcal{F}_t]$. We shall often omit to write \lq\lq almost surely\rq\rq\ (when a property holds outside a $P$-null set). Let $\T_t$ be the set of $\mathbb{F}$-stopping times taking values in $\{t\dotso,T\}$, and $\T:=\T_0$. We define the stopping strategies of Type I and Type II as follows:
\begin{definition}
$\Rho$ is a stopping strategy of Type I (resp. II), if $\Rho:\ \T\mapsto\T$ satisfies the \lq\lq non-anticipativity\rq\rq\ condition of Type I (resp. II), i.e., for any $\sigma_1,\sigma_2\in\T$,
\begin{equation}\label{e1}
\text{either}\quad\Rho(\sigma_1)=\Rho(\sigma_2)\leq\text{(resp. $<$) }\sigma_1\wedge\sigma_2\quad\text{or}\quad \Rho(\sigma_1)\wedge\Rho(\sigma_2)>\text{(resp. $\geq$) }\sigma_1\wedge\sigma_2.
\end{equation}
Denote $\bT^i$ (resp. $\bT^{ii}$) as the set of stopping strategies of Type I (resp. II).
\end{definition}
\begin{remark}
We can treat $\T$ as a subset of $\bT^i$ and $\bT^{ii}$ (i.e., each $\tau\in\T$ can be treated as the map with only one value $\tau$). Hence we have $\T\subset\bT^i\subset\bT^{ii}$.
\end{remark}

Consider the problem
\begin{equation}\label{e2}
\overline V:=\inf_{\Rho\in\bT^{ii}}\sup_{\tau\in\T}\E[U(\Rho(\tau),\tau)]\quad\text{and}\quad\underline V:=\sup_{\Tau\in\bT^i}\inf_{\rho\in\T}\E[U(\rho,\Tau(\rho))].
\end{equation}
We shall convert this problem into a Dynkin game. In order to do so, let us introduce the following two processes that will represent the payoffs in the Dynkin game. Let
\begin{equation}\label{e3}
V_t^1:=\essinf_{\rho\in\T_t}\E_t[U(\rho,t)],\quad t=0,\dotso,T,
\end{equation}
and
\begin{equation}\label{e13}
V_t^2:=\max\left\{\esssup_{\tau\in\T_{t+1}}\E_t[U(t,\tau)], V_t^1\right\},\quad t=0,\dotso,T-1,
\end{equation}
and $V_T^2=U(T,T)$. Observe that
\begin{equation}\label{e4}
V_t^1\leq V_t^2,\quad t=0,\dotso,T.
\end{equation}
By the classic optimal stopping theory (see e.g., \cite[Appendix D]{MR1640352}), there exist an optimizer $\rho_u(t)\in\T_t$ for $V_t^1$, and an optimizer $\tau_u(t)\in\T_{t+1}$ for $\esssup_{\tau\in\T_{t+1}}\E_t[U(t,\tau)]$, $t=0,\dotso,T-1$. We let  $\rho_u(T)=\tau_u(T)=T$ for convenience. 


Define the corresponding Dynkin game as follows:
\begin{equation}\notag
V:=\inf_{\rho\in\T}\sup_{\tau\in\T}\E\left[V_\tau^1 1_{\{\tau
\leq\rho\}}+V_\rho^2 1_{\{\tau>\rho\}}\right]=\sup_{\tau\in\T}\inf_{\rho\in\T}\E\left[V_\tau^1 1_{\{\tau\leq\rho\}}+V_\rho^2 1_{\{\tau>\rho\}}\right],
\end{equation}
where the second equality above follows from \eqref{e4}. Moreover, there exists a saddle point $(\rho_d,\tau_d)$ (see e.g., \cite{zbMATH06169716}) described by
\begin{equation}\label{e6}
\rho_d:=\inf\{s\geq 0:\ V_s=V_s^2\}\quad\text{and}\quad \tau_d:=\inf\{s\geq 0:\ V_s=V_s^1\},
\end{equation}
where
$$V_t:=\essinf_{\rho\in\T_t}\esssup_{\tau\in\T_t}\E_t\left[V_\tau^1 1_{\{\tau
\leq\rho\}}+V_\rho^2 1_{\{\tau>\rho\}}\right]=\esssup_{\tau\in\T_t}\essinf_{\rho\in\T_t}\E_t\left[V_\tau^1 1_{\{\tau\leq\rho\}}+V_\rho^2 1_{\{\tau>\rho\}}\right].$$
That is,
$$V=\sup_{\tau\in\T}\E\left[V_\tau^1 1_{\{\tau
\leq\rho_d\}}+V_{\rho_d}^2 1_{\{\tau>\rho_d\}}\right]=\inf_{\rho\in\T}\E\left[V_{\tau_d}^1 1_{\{\tau_d\leq\rho\}}+V_\rho^2 1_{\{\tau_d>\rho\}}\right].$$

Below is the main result of this paper. 
\begin{theorem}\label{t1}
We have that
$$\overline V=\underline V=V.$$
Besides, there exists $\Rho^*\in\bT^{ii}$ and $\tau^*:\bT^{ii}\mapsto\T$ described by 
\begin{equation}\label{e5}
\Rho^*(\tau)=\rho_d 1_{\{\tau>\rho_d\}}+\rho_u(\tau) 1_{\{\tau\leq\rho_d\}},\quad\tau\in\T,
\end{equation}
and
\begin{equation}\label{e8}
\tau^*(\Rho):=\tau_d 1_{\{\tau_d
\leq\Rho(\tau_d)\}}+\tau_u(\Rho(\tau_d)) 1_{\{\tau_d>\Rho(\tau_d)\}},\quad\Rho\in\bT^{ii},
\end{equation}
such that
$$\overline V=\sup_{\tau\in\T}\E[U(\Rho^*(\tau),\tau)]=\inf_{\Rho\in\bT^{ii}}\E[U(\Rho(\tau^*(\Rho)),\tau^*(\Rho))].$$
Similarly, there exists $\Tau^{**}\in\bT^i$ and $\rho^{**}:\bT^i\mapsto\T$ described by 
\begin{equation}\label{e17}
\Tau^{**}(\rho)=\tau_d 1_{\{\rho\geq\tau_d\}}+\tau_u(\rho) 1_{\{\rho<\tau_d\}},\quad\rho\in\T,
\end{equation}
and
\begin{equation}\notag
\rho^{**}(\Tau):=\rho_d 1_{\{\rho_d
<\Tau(\rho_d)\}}+\rho_u(\Tau(\rho_d)) 1_{\{\rho_d\geq\Tau(\rho_d)\}},\quad\Tau\in\bT^i,
\end{equation}
such that
$$\underline V=\inf_{\rho\in\T}\E[U(\rho,\Tau^{**}(\rho))]=\sup_{\Tau\in\bT^i}\E[U(\rho^{**}(\Tau),\Tau(\rho^{**}(\Tau)))].$$
\end{theorem}
\begin{remark}\label{r1}
As the inner player in $\overline V$, $\tau$ depends on $\Rho$. Therefore, as a good reaction to $\Rho$, $\tau^*(\cdot)$ defined in \eqref{e8} is a map from $\bT^{ii}$ to $\T$ instead of a stopping time. (To convince oneself, one may think of $\inf_x\sup_y f(x,y)=\inf_x f(x,y^*(x))$.)
\end{remark}
\begin{corollary}
$$\overline V=\E[U(\Rho^*(\tau^*(\Rho^*)),\tau^*(\Rho^*))].$$
Moreover,
\begin{equation}\label{e16}
\Rho^*(\tau^*(\Rho^*))=\rho_d 1_{\{\tau_d>\rho_d\}}+\rho_u(\tau_d) 1_{\{\tau_d
\leq\rho_d\}}\quad\text{and}\quad\tau^*(\Rho^*)=\tau_d 1_{\{\tau_d
\leq\rho_d\}}+\tau_u(\rho_d) 1_{\{\tau_d>\rho_d\}}.
\end{equation}
Similar results hold for $\underline V$.
\end{corollary}
\begin{proof}
By \eqref{e5},
$$\Rho^*(\tau_d)=\rho_d 1_{\{\tau_d>\rho_d\}}+\rho_u(\tau_d) 1_{\{\tau_d\leq\rho_d\}}.$$
If $\tau_d>\rho_d$, then $\Rho^*(\tau_d)=\rho_d<\tau_d$, which implies that $\{\tau_d>\rho_d\}\subset\{\tau_d>\Rho^*(\tau_d)\}$. If $\tau_d\leq\rho_d$, then $\Rho^*(\tau_d)=\rho_u(\tau_d)\geq\tau_d$, which implies that $\{\tau_d\leq\rho_d\}\subset\{\tau_d\leq\Rho^*(\tau_d)\}$. Therefore, $\{\tau_d>\rho_d\}=\{\tau_d>\Rho^*(\tau_d)\}$ and $\{\tau_d\leq\rho_d\}=\{\tau_d\leq\Rho^*(\tau_d)\}$. Hence we have that
$$\tau^*(\Rho^*)=\tau_d 1_{\{\tau_d
\leq\rho_d\}}+\tau_u(\Rho^*(\tau_d)) 1_{\{\tau_d>\rho_d\}}=\tau_d 1_{\{\tau_d
\leq\rho_d\}}+\tau_u(\rho_d) 1_{\{\tau_d>\rho_d\}},$$
where the second equality follows from that $\Rho^*(\tau_d)=\rho_d$ on $\{\tau_d>\rho_d\}$.

Now if $\tau_d\leq\rho_d$, then $\tau^*(\Rho^*)=\tau_d\leq\rho_d$, and thus $\{\tau_d\leq\rho_d\}\subset\{\tau^*(\Rho^*)\leq\rho_d\}$. If $\tau_d>\rho_d$, then $\tau^*(\Rho^*)=\tau_u(\rho_d)>\rho_d$ since $\tau_u(t)\geq t+1$ if $t<T$, and thus $\{\tau_d>\rho_d\}\subset\{\tau^*(\Rho^*)>\rho_d\}$. Therefore, $\{\tau_d\leq\rho_d\}=\{\tau^*(\Rho^*)\leq\rho_d\}$ and $\{\tau_d>\rho_d\}=\{\tau^*(\Rho^*)>\rho_d\}$. Hence we have that
$$\Rho^*(\tau^*(\Rho^*))=\rho_d 1_{\{\tau_d>\rho_d\}}+\rho_u(\tau^*(\Rho^*)) 1_{\{\tau_d\leq\rho_d\}}=\rho_d 1_{\{\tau_d>\rho_d\}}+\rho_u(\tau_d) 1_{\{\tau_d\leq\rho_d\}},$$
where the second equality follows from that $\tau^*(\Rho^*)=\tau_d$ on $\{\tau_d\leq\rho_d\}$.
\end{proof}

\section{Examples}
In this section we provide three examples that fall within the setup of Section 2. The first example shows that in the classical Dynkin game one does not need to use non-anticipative stopping strategies. The second example is a relevant problem from mathematical finance in which our results can be applied.  This problem is on determining the optimal exercise strategy when one trades two different American options in different directions. In the third example we consider two competing companies making an entry decision into a particular market.

\subsection{Dynkin game using non-anticipative stopping strategies} Let
$$U(s,t)=f_s 1_{\{s<t\}}+g_t 1_{\{s\geq t\}},$$
where $(f_t)_t$ and $(g_t)_t$ are $\mathbb{F}$-adapted, satisfying $f\geq g$. Then we have that 
$$V_t^1=g_t,\ t=0,\dotso,T,\quad\text{and}\quad V_t^2=f_t,\ t=0,\dotso,T-1.$$
Then by \thref{t1} we have that
\begin{eqnarray}
\notag &&\inf_{\Rho\in\bT^{ii}}\sup_{\tau\in\T}\E\left[f_{\Rho(\tau)} 1_{\{\Rho(\tau)<\tau\}}+g_\tau 1_{\{\Rho(\tau)\geq\tau\}}\right]=\sup_{\Tau\in\bT^i}\inf_{\rho\in\T}\E\left[f_\rho 1_{\{\rho<\Tau(\rho)\}}+g_{\Tau(\rho)} 1_{\{\rho\geq\Tau(\rho)\}}\right]\\
\notag&&=\sup_{\tau\in\T}\inf_{\rho\in\T}\E\left[f_\rho 1_{\{\rho<\tau\}}+g_\tau 1_{\{\rho\geq\tau\}}\right]=\inf_{\rho\in\T}\sup_{\tau\in\T}\E\left[f_\rho 1_{\{\rho<\tau\}}+g_\tau 1_{\{\rho\geq\tau\}}\right].
\end{eqnarray}
Besides, by the property of $U$, the $\Rho^*$ and $\Tau^{**}$ defined in \eqref{e5} and \eqref{e17} can w.l.o.g. be written as
$$\Rho=\rho_d\quad\text{and}\quad\Tau=\tau_d.$$
Therefore, in the Dynkin game, using  non-anticipative stopping strategies is the same as using a usual stopping time. 

\begin{remark}
In this example we let $\Rho\in\bT^{ii}$ and $\Tau\in\bT^i$. The same conclusion holds if we let  $\Rho\in\bT^i$ and $\Tau\in\bT^{ii}$ instead.
\end{remark}

\subsection{A robust utility maximization problem} Let
$$U(t,s)=\mathcal{U}(f_t-g_s),$$
where $\mathcal{U}: \mathbb{R}\mapsto\mathbb{R}$ is a utility function, and $f$ and $g$ are adapted to $\mathbb{F}$. Consider
$$\overline{\mathcal{V}}:=\sup_{\Rho\in\bT^{ii}}\inf_{\tau\in\T}\E[U(\Rho(\tau),\tau)].$$
This problem can be interpreted as the one in which an investor longs an American option $f$ and shorts an American option $g$, and the goal is to choose an optimal stopping strategy to maximize the utility according to the stopping behavior of the holder of $g$. Here we assume that the maturities of $f$ and $g$ are the same (i.e., $T$). This is without loss of generality. Indeed for instance, if the maturity of $f$ is $\hat t<T$, then we can define $f(t)=f(\hat t)$ for $t=\hat t+1,\dotso,T$.

\subsection{Time to enter the market}
There are two companies choosing when to enter a specific market. These two companies will produce the same kind of product. The one that enters the market first can start collecting profit earlier, while the one that enters second can use the other's experience (e.g., marketing strategies, technologies) to reduce its own cost. Hence, each company's entering time will affect the profit and market share no matter if it enters the market first or second.  Then a natural and simple model for this set-up would be given by our game defined in \eqref{e2}.

\section{Proof of \thref{t1}}
We will only prove the results for $\overline V$, since the proofs for $\underline V$ are similar.
\begin{lemma}\label{l1}
For any $\sigma\in\T$, $\rho_u(\sigma)\in\T$ and $\tau_u(\sigma)\in\T$. 
\end{lemma}
\begin{proof}
Take $\sigma\in\T$. Then for $t\in\{0,\dotso,T\}$
$$\{\rho_u(\sigma)\leq t\}=\cup_{i=0}^t(\{\sigma=i\}\cap\{\rho_u(i)\leq t\})\in\mathcal{F}_t.$$
\end{proof}
\begin{lemma}\label{l3}
$\Rho^*$ defined in \eqref{e5} is in $\bT^{ii}$ and $\tau^*$ defined in \eqref{e8} is a map from $\bT^{ii}$ to $\T$.
\end{lemma}
\begin{proof}
Take $\tau\in\T$. We have that
\begin{eqnarray}
\notag \{\Rho^*(\tau)\leq t\}&=&(\{\tau>\rho_d\}\cap\{\rho_d\leq t\})\cup(\{\tau
\leq\rho_d\}\cap\{\rho_u(\tau)\leq t\})\\
\notag &=&(\{\tau>\rho_d\}\cap\{\rho_d\leq t\})\cup(\{\tau
\leq\rho_d\}\cap\{\tau\leq t\}\cap\{\rho_u(\tau)\leq t\})\in\mathcal{F}_t.
\end{eqnarray}
Hence $\Rho^*(\tau)\in\T$. Similarly we can show that $\tau^*(\Rho)\in\T$ for any $\Rho\in\bT^{ii}$.

It remains to show that $\Rho^*$ satisfies the non-anticipative condition of Type II in \eqref{e1}. Take $\tau_1,\tau_2\in\T$. If $\Rho^*(\tau_1)<\tau_1\wedge\tau_2\leq\tau_1$, then $\tau_1>\rho_d$ and thus $\Rho^*(\tau_1)=\rho_d<\tau_1\wedge\tau_2\leq\tau_2$, which implies $\Rho^*(\tau_2)=\rho_d=\Rho^*(\tau_1)<\tau_1\wedge\tau_2$. If $\Rho^*(\tau_1)\geq\tau_1\wedge\tau_2$, then if $\Rho^*(\tau_2)<\tau_1\wedge\tau_2$ we can use the previous argument to get that $\Rho^*(\tau_1)=\Rho^*(\tau_2)<\tau_1\wedge\tau_2$ which is a contradiction, and thus $\Rho^*(\tau_2)\geq\tau_1\wedge\tau_2$.
\end{proof}

\begin{lemma}
$$\overline V\leq \sup_{\tau\in\T}\E[U(\Rho^*(\tau),\tau)]\leq V.$$
\end{lemma}
\begin{proof}
Recall $\Rho^*$ defined in \eqref{e5} and $\rho_d$ defined in \eqref{e6}. We have that
\begin{eqnarray}
\notag \overline V&\leq& \sup_{\tau\in\T}\E[U(\Rho^*(\tau),\tau)]\\
\notag &=& \sup_{\tau\in\T}\E\left[U(\rho_d,\tau) 1_{\{\rho_d<\tau\}}+U(\rho_u(\tau),\tau) 1_{\{\rho_d\geq\tau\}}\right]\\
\notag &=& \sup_{\tau\in\T}\E\left[1_{\{\rho_d<\tau\}}\E_{\rho_d}[U(\rho_d,\tau)] +1_{\{\rho_d\geq\tau\}}\E_\tau[U(\rho_u(\tau),\tau)] \right]\\
\notag&\leq&\sup_{\tau\in\T}\E\left[1_{\{\rho_d<\tau\}}V_{\rho_d}^2+1_{\{\rho_d
\geq\tau\}}V_\tau^1\right]\\
\notag&=&V.
\end{eqnarray}
\end{proof}

\begin{lemma}\label{l2}
$$\overline V\geq\inf_{\Rho\in\bT^{ii}}\E[U(\Rho(\tau^*(\Rho)),\tau^*(\Rho))]\geq V.$$
\end{lemma}
\begin{proof}
Take $\Rho\in\bT^{ii}$. Recall $\tau^*$ defined in \eqref{e8}. By the non-anticipativity condition of Type II in \eqref{e1}, 
$$\text{either}\quad\Rho(\tau^*(\Rho))=\Rho(\tau_d)<\tau_d\wedge\tau^*(\Rho)\quad\text{or}\quad\Rho(\tau^*(\Rho))\wedge\Rho(\tau_d)\geq\tau_d\wedge\tau^*(\Rho).$$
Therefore, 
$$\text{if}\quad\Rho(\tau_d)\geq\tau_d,\quad{then}\quad\Rho(\tau^*(\Rho))\geq\tau_d\wedge\tau^*(\Rho)=\tau_d=\tau^*(\Rho),$$
and
\begin{eqnarray}
\text{if}\quad\Rho(\tau_d)<\tau_d,\quad{then}\quad\tau^*(\Rho)=\tau_u(\Rho(\tau_d))>\Rho(\tau_d)&\implies&\Rho(\tau_d)<\tau^*(\Rho)\wedge\tau_d\label{e15}\\
&\implies&\Rho(\tau_d)=\Rho(\tau^*(\Rho)),\notag
\end{eqnarray}
where in \eqref{e15} we used the fact that $\tau_u(t)\geq t+1$ if $t<T$ (in the first conclusion).

Besides, if $\tau_d>\Rho(\tau_d)$, then by the fact that $V^1\leq V$ and \eqref{e6} we have that
$$V_{\Rho(\tau_d)}^1<V_{\Rho(\tau_d)}\leq V_{\Rho(\tau_d)}^2,$$
which implies that
$$V_{\Rho(\tau_d)}^2=\esssup_{\tau\in\T_{\Rho(\tau_d)+1}}\E_{\Rho(\tau_d)}[U(\Rho(\tau_d),\tau)]=\E_{\Rho(\tau_d)}[U(\Rho(\tau_d),\tau_u(\Rho(\tau_d)))].$$

Now we have that
\begin{eqnarray}
\notag\sup_{\tau\in\T}\E[U(\Rho(\tau),\tau))]&\geq&\E[U(\Rho(\tau^*(\Rho)),\tau^*(\Rho))]\\
\notag&=&\E\left[U(\Rho(\tau^*(\Rho)),\tau^*(\Rho)) 1_{\{\tau_d\leq\Rho(\tau_d)\}}+U(\Rho(\tau^*(\Rho)),\tau^*(\Rho)) 1_{\{\tau_d>\Rho(\tau_d)\}}\right]\\
\notag&=&\E\left[U(\Rho(\tau^*(\Rho)),\tau_d) 1_{\{\tau_d
\leq\Rho(\tau_d)\}}+U(\Rho(\tau_d),\tau_u(\Rho(\tau_d))) 1_{\{\tau_d>\Rho(\tau_d)\}}\right]\\
\notag&=&\E\left[1_{\{\tau_d
\leq\Rho(\tau_d)\}}\E_{\tau_d}[U(\Rho(\tau^*(\Rho)),\tau_d)] +1_{\{\tau_d>\Rho(\tau_d)\}}\E_{\Rho(\tau_d)}[U(\Rho(\tau_d),\tau_u(\Rho(\tau_d)))]\right]\\
\notag&\geq&\E\left[1_{\{\tau_d
\leq\Rho(\tau_d)\}}V_{\tau_d}^1+1_{\{\tau_d>\Rho(\tau_d)\}}V_{\Rho(\tau_d)}^2\right]\\
\notag&\geq&\inf_{\rho\in\T}\E\left[1_{\{\tau_d
\leq\rho\}}V_{\tau_d}^1+1_{\{\tau_d>\rho\}}V_\rho^2\right]\\
\notag&=&V,
\end{eqnarray}
where the fifth inequality follows from the definition of $V^1$ in \eqref{e3} and the fact that $\Rho(\tau^*(\Rho))\geq\tau_d$ on $\{\Rho(\tau_d)\geq\tau_d\}$. As this holds for arbitrary $\Rho\in\bT^{ii}$, the conclusion follows.
\end{proof}

\begin{proof}[\textbf{Proof of \thref{t1}}]
This follows from Lemmas \ref{l1}-\ref{l2}.
\end{proof}

\section{Some insight into the continuous-time version}
We can also consider the continuous time version of the stopper-stopper problem. If we want to follow the argument in Section 4, there are mainly two technical parts we need to handle as opposed to the discrete-time case, which are as follows.
\begin{itemize}
\item We need to make sure that $V^1$ and $V^2$ defined in \eqref{e3} and \eqref{e13} have RCLL modifications.
\item On an intuitive level, the optimizers (or choose to be $\eps$-optimizers in continuous time) $\rho_u(\cdot)$ and $\tau_u(\cdot)$ are maps from $\T$ to $\T$. Yet this may not be easy to prove in continuous time, as opposed to the argument in \leref{l3}.
\end{itemize}
In order to address the two points above, we may have to assume some continuity of $U$ in $(s,t)$ (maybe also in $\omega$). On the other hand, with such continuity, there will essentially be no difference between using stopping strategies of Type I and using stopping strategies of Type II, as opposed to the discrete-time case (see Examples \ref{g1} and \ref{g2}).

\bibliographystyle{siam}
\bibliography{ref}

\end{document}